\documentclass[letterpaper,11pt,twoside]{article}
\usepackage{times}
\usepackage{amsmath}
\usepackage{amssymb}
\usepackage{amsfonts}
\usepackage{amscd}
\usepackage{amsthm}
\usepackage{amsxtra}
\usepackage{mathtools}
\usepackage{stmaryrd}
\usepackage[all]{xy}  \CompileMatrices
\usepackage{mathrsfs}
\usepackage{nicefrac}
\usepackage{graphicx}
\usepackage{color}
\usepackage[raggedright,IT,hang]{subfigure}
\usepackage{wrapfig}
\usepackage{enumerate}
\usepackage{textcomp}
\usepackage{float}

\usepackage{tikz}
\usetikzlibrary{intersections}

\newcommand{\R}{\mathbb{R}}
\DeclareMathOperator{\affOp}{aff}
\newcommand{\aff}{\affOp}
\DeclareMathOperator{\idOp}{id}
\newcommand{\id}{\idOp}
\DeclareMathOperator{\stabOp}{STAB}
\newcommand{\stab}{\stabOp}
\DeclareMathOperator{\tsOp}{TS}
\newcommand{\ts}{\tsOp}

\DeclareMathOperator{\vertOp}{vert}
\newcommand{\vertex}{\vertOp}

\DeclareMathOperator{\convOp}{conv}
\newcommand{\conv}{\convOp}

\newcommand{\setDef}[2]{\{{#1}:{#2}\}}

\newtheorem{prop}{Proposition}
\newtheorem{remark}[prop]{Remark}
\newtheorem{lem}[prop]{Lemma}
\newtheorem{obs}[prop]{Observation}
\newtheorem{thm}[prop]{Theorem}

\begin{document}

\begin{center}\LARGE The Projected Faces Property and Polyhedral Relations.
\\[\baselineskip]
\large Michele Conforti\footnote{Dipartimento di Matematica, Universit\`a degli Studi di Padova, Via Trieste 63, 35121 Padova, Italy ({\tt conforti@math.unipd.it}). Supported by ``Progetto di Eccellenza 2008--2009'' of ``Fondazione Cassa di Risparmio di Padova e Rovigo''.},
\large Kanstantsin Pashkovich\footnote{Universit\'{e} libre de Bruxelles, D\'{e}partement de Math\'{e}matique, Boulevard du Triomphe, B-1050 Brussels, Belgium ({\tt kanstantsin.pashkovich@gmail.com}). Supported by ``Progetto di Eccellenza 2008--2009'' of ``Fondazione Cassa di Risparmio di Padova e Rovigo'' and Fonds de la Recherche Scientifique de Belgique (FRS-FNRS).}
\\[\baselineskip]
\today
\end{center}

\begin{abstract}
Margot (1994) in his doctoral dissertation studied extended formulations of combinatorial polytopes  that arise from "smaller" polytopes via some composition rule.  He introduced the "projected faces property" of a polytope and showed that this property suffices  to iteratively build extended formulations of composed polytopes.

For the composed polytopes, we show that an extended formulation of the type studied in this paper is always possible only if the smaller polytopes have the projected faces property. Therefore, this produces a characterization of the projected faces property.

  Affinely generated polyhedral relations were introduced by  Kaibel and Pashkovich (2011) to construct extended formulations for the convex hull of the images of a point under the action of some finite group of reflections. In this paper we prove that the projected faces property and affinely generated polyhedral relation are equivalent conditions.

%  {\color{red} Moreover, in this paper we show that the qualities implied by the projected faces property, which the iterative framework developed by Margot was built on, form a characterization of the projected faces property. }
\end{abstract}

\section{Introduction}

\iffalse
A \emph{polytope}~$P\subseteq\R^n$ is a convex hull of a finite (maybe empty) set  of points in $\R^n$. An inclusion wise minimum set~$\vertex(P)$ of points defining the polytope~$P$ is the set of its \emph{vertices}.  Additionally, every polytope can be described by a linear system of inequalities and equations. This representation of the polytope~$P$ allows to use the machinery of linear programming to optimize a linear function over~$\vertex(P)$. For other notions and results from polyhedral theory we refer the reader to~\cite{GR}.

However, the number of inequalities needed to describe the polytope~$P$ may be very big, in this case one may consider using extended formulations. A polytope~$Q\subseteq\R^d$ together with an affine map $p:\R^d\rightarrow\R^n$ is an \emph{extension} of the polytope~$P$ if $p(Q)$ equals~$P$; its \emph{size} is the number of facets of the polytope $Q$, i.e. a minimum number of inequalities in a linear description of~$Q$.

Extended formulations attracted a lot of attention in the current past. The idea can be seen as an extension of linear models by introducing additional variables, which in many cases leads to reduction in terms of the problem's size. For an extensive overview on extended formulation we refer the reader to~\cite{CCZ}.

\fi

Margot \cite{FM} in his doctoral dissertation studied extended formulations for combinatorial polytopes defined on classes of graphs that are closed under some composition rule.

There are not many paradigms to construct extended formulations, see \cite{CCZ}, \cite{K} for a survey. So the main purpose of this paper is to introduce, illustrate with examples and analyze the framework constructed by Margot.
%Apart from presenting the ideas and results from the dissertation of  Margot, we strengthen Corollary~2.11 in his dissertation providing thus a characterization of the projected faces property (or shortly PF-property) and establish {\color{red} equivalence between the PF-property and affinely generated polyhedral relation}, another framework developed by  Kaibel and Pashkovich in~\cite{KP}. {\color{red} We also illustrate PF-property by explicitly constructing extentded formulations.}
 \smallskip

 We refer the reader to \cite{SC} for polyhedral theory. Let~$P\subseteq\R^{n}\times\R^{d}$ be a polytope. Its {\em projection} into $\R^{n}$ is the polytope  $p_x(P):=\{x\in \R^n:\; \exists y\in \R^d \text{ such that }\,(x,y)\in P\}$,  where $p_x:\R^{n}\times\R^{d}\rightarrow \R^n$ is the projection map into $\R^{n}$, i.e. in the $x$~variables.
Note that $p_x(P)=\conv(\setDef{p_x(v)}{v\in\vertex(P)})$, where~$\vertex(P)$ denotes the vertex set of a polytope~$P$.

 An {\em extended formulation} of a polytope $P\subseteq\R^{n}$ is a system of inequalities $Ax+By\le d$ that defines a polytope $Q\subseteq\R^{n}\times\R^{d}$ such that $P=p_x(Q)$. The polytope $Q\subseteq\R^{n}\times\R^{d}$ together with $p_x:\R^{n}\times\R^{d}\rightarrow \R^{n}$ is called an \emph{extension} of $P$.   Since
 $$\max \{cx:\; x\in P\}=\max\{ cx+\mathbf{0}y:\; Ax+By\le d\}\,,$$
an extended formulation of small size allows to optimize a linear function over $P$ by solving a linear optimization problem involving a  small number of linear inequalities.
 \smallskip

In his dissertation,  Margot addresses the following problem:
\begin{equation}\label{prob:Margot}
\begin{multlined}
 \text{\parbox{.83\textwidth}{\raggedright
    {\it Given two polytopes $P_1\subseteq\R^{n_1}\times\R^{d_1}$ and $P_2\subseteq \R^{n_2}\times\R^{d_2}$ and a function $f:\R^{n_1}\times\R^{n_2}\rightarrow\R^n\cup\{\infty\}$, provide a linear description of the polytope $P$ which is  the convex hull of the following set of points}}}\\
    (\gamma ,x,y)\in\R^n\times \R^{d_1}\times\R^{d_2}   \text{\it such that } (\alpha,x)\in \vertex(P_1),\\ (\beta ,y)\in \vertex(P_2) \, \text{\it  and }
    \gamma =f(\alpha,\beta )\neq \infty\\ \text{\it for some  } \alpha\in\R^{n_1}, \beta \in\R^{n_2}\,.
    \end{multlined}
\end{equation}

\smallskip

%The approach of Margot to solve the problem~\eqref{prob:Margot} is to interpret each of the conditions in the problem in a polyhedral way: to relax $(\alpha,x)\in\vertex(P_1)$, $(\beta,y)\in\vertex (P_2)$ to the conditions $(\alpha,x)\in P_1$, $(\beta,y)\in P_2$, and to express the rest of the conditions using the third polytope $P_3\subseteq \R^n\times \R^{n_1}\times\R^{n_2}$, which is defined as follows

\smallskip

A polytope $P\subseteq \R^n\times \R^d$,  together with the projection map~$p_x:\R^n\times \R^d\rightarrow \R^n$  has the {\em PF-property} (projected faces property) if every face of the polytope~$P$ is projected to a face of the polytope~$p_x(P)$. That is  for every face~$F\subseteq \R^{n}\times\R^{d}$ of the polytope~$P$ the projection~$p_x(F)\subseteq\R^n$ is a face of the polytope~$p_x(P)$ (see Figure~\ref{fig:pfp}).
\begin{figure}[H]
  \begin{center}
    \begin{tikzpicture}[scale=1.4]
      \foreach \x/\y in {0/0,70/1,100/2,180/3,220/4,300/5} {\coordinate (R-\y) at ({cos(\x)}, {sin(\x)*.3+.5*cos(\x)+.7});};
      \foreach \y  in {0,...,5} { \pgfmathparse{mod(\y+1,6)} \draw (R-\y)--(R-\pgfmathresult);}
      \foreach \x/\y in {0/0,70/1,100/2,180/3,220/4,300/5} {\coordinate (L-\y) at ({cos(\x)}, {sin(\x)*.3-2});};
      \foreach \y  in {0,...,5} { \pgfmathparse{mod(\y+1,6)} \draw (L-\y)--(L-\pgfmathresult); \draw (L-\y)--(R-\y);}
      \draw[dashed, white] (L-1)--(R-1);
      \draw[dashed, white] (L-2)--(R-2);
      \draw[dashed, white] (L-0)--(L-1)--(L-2)--(L-3);
    \end{tikzpicture}
    \hspace{1cm}
    \begin{tikzpicture}[scale=1.4]
      \foreach \x/\y in {0/0,70/1,100/2,180/3,220/4,300/5} {\coordinate (R-\y) at ({cos(\x)}, {sin(\x)*.3});};
      \foreach \x/\y in {0/0,100/2,220/4} {\coordinate (R-\y) at ({cos(\x)}, {sin(\x)*.3+1});};
      \foreach \y  in {0,...,5} { \pgfmathparse{mod(\y+1,6)} \draw (R-\y)--(R-\pgfmathresult);}
      \draw[very thick] (R-0)--(R-2)--(R-4)--(R-0);
      \foreach \x/\y in {0/0,70/1,100/2,180/3,220/4,300/5} {\coordinate (L-\y) at ({cos(\x)}, {sin(\x)*.3-2});};
      \foreach \y  in {0,...,5} { \pgfmathparse{mod(\y+1,6)} \draw (L-\y)--(L-\pgfmathresult); \draw (L-\y)--(R-\y);
      \draw[dashed, white] (L-1)--(R-1);
      \draw[dashed, white] (L-2)--(R-2);
      \draw[dashed, white] (L-0)--(L-1)--(L-2)--(L-3);
      \draw[dashed, white] (R-0)--(R-1)--(R-2)--(R-3);
}
    \end{tikzpicture}
    \label{fig:pfp}
    \caption{The left three-dimensional polytope together with its projection on the space defining the lower facet possesses the PF-property, the right three-dimensional polytope with the corresponding projection does not possess the PF-property; for example, the triangle facet above is not projected to any face of the projection.}
  \end{center}
\end{figure}
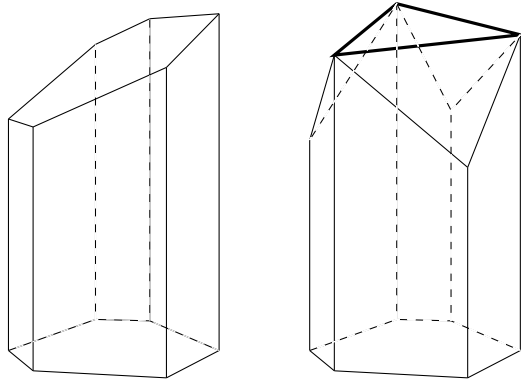

\smallskip
Let us define the following polytope
\iffalse
\begin{equation}\label{def:P3}
\begin{multlined}
\fi
\begin{multline*}
  P_3:=\conv(  (\gamma,\alpha,\beta)\in \R^n\times \R^{n_1}\times\R^{n_2} \text{ such that } \\\alpha \in\vertex(p_{\alpha}(P_1))\, , \beta \in\vertex(p_{\beta}(P_2))\,,\gamma=f(\alpha,\beta)\text{  and  } \gamma \neq \infty)\,.
\end{multline*}
\iffalse
\end{multlined}
\end{equation}
\fi
Thus, $\vertex(P_3)$ represent all pairs of a vertex $\alpha\in\vertex(p_{\alpha}(P_1))$ and a vertex $\beta \in\vertex(p_{\beta}(P_2))$, which are feasible for $f:\R^{n_1}\times\R^{n_2}\rightarrow\R^n\cup\{\infty\}$, i.e. $f(\alpha,\beta)\neq\infty$.  Additionally, to every vertex of $P_3$ the corresponding $f(\alpha, \beta)$ is assigned.
%   Given polytopes  $P_1=\{(\alpha,x)\in  \R^{n_1}\times\R^{d_1}:\,A_1(\alpha, x)\le b_1\}$ and  $P_2=\{(\beta, y)\in  \R^{n_2}\times\R^{d_2}:\,A_2(\beta,y)\le b_2\}$ and  $P_3=A_3 (\gamma,\alpha, \beta)\le b_3$ be a linear description of the polytope

Margot~\cite{FM} shows the following:

\begin{thm}\label{thm:margot_theorem} Given systems $A_1(\alpha, x)\le b_1$, $A_2(\beta,y)\le b_2$ and $A_3 (\gamma,\alpha, \beta)\le b_3$ defining the polytopes  $P_1$, $P_2$ and $P_3$  as  above,
let $Q:=\{(\gamma,\alpha,x,\beta, y)\in  \R^n\times\R^{n_1}\times\R^{d_1}\times \R^{n_2}\times\R^{d_2}:\;A_1(\alpha,x)\le b_1, A_2(\beta,y)\le b_2, A_3 (\gamma,\alpha, \beta)\le b_3\}$ and let $P$ be as defined in \eqref{prob:Margot}.

If both~$(P_1,p_{\alpha})$  and~$(P_2, p_{\beta})$  have  the PF-property and every vertex of~$P_3$ projects into a vertex of~$p_{\gamma}(P_3)$, then
\begin{enumerate}[a)]
  \item \label{enum:projection}  $P=p_{\gamma, x ,y}(Q)$.
  \item \label{enum:pfp} $(Q, p_{\gamma,\alpha,\beta})$ has the PF-property.
\end{enumerate}
\end{thm}

After illustrating with examples in Section~\ref{sec:examples}, in Section~\ref{sec:characterization} we prove our main result, namely that the PF-property is both necessary and sufficient for Theorem~\ref{thm:margot_theorem} to hold for every choice of function $f$ as defined in Problem~\eqref{prob:Margot}.
 In Section~\ref{sec:framework} we prove Theorem~\ref{thm:margot_theorem} and discuss iterative use of this theorem. Finally, in Section~\ref{sec:affine_generated_relations} we show that the PF-property  is an alternative way to see affinely generated polyhedral relations defined in~\cite{KP}.

\section{Combinatorial Polytopes and PF-Property}\label{sec:examples}

The following examples illustrate the role of the problem~\eqref{prob:Margot} and the PF-property with respect to extended formulations of combinatorial polytopes.

\subsection{Parity Polytope}

The \emph{parity polytope~$Q_n$} is the convex hull of the $n$-dimensional $0,1$-points with an even number of coordinates equal to~$1$. For $n\ge 4$ the polytope~$Q_n$ can be described by the linear system below~\cite{Je}
\begin{equation*}
\sum_{i \in S} x_i -\sum_{\substack{i \not \in S\\ i\in \{1,\ldots,n\}}} x_i \le |S|-1\qquad \text{ for } S\subseteq \{1,\ldots, n\},\, |S|\text{ odd}
\end{equation*}
and $\mathbf{0}\le x\le \mathbf{1}$.
Furthermore, all $2^{n-1}+2n$ inequalities in the above description induce facets of $Q_n$~\cite{Je}, thus every linear description of~$Q_n$ in the initial space involves an exponential number of inequalities. Nevertheless there is an extended formulation for the parity polytope~$Q_n$ of size~$4(n-1)$~\cite{CK}.

Let us iteratively construct an extension for the parity polytope starting with $Q_1=\{(0)\}$ and using Theorem~\ref{thm:margot_theorem}. For this let us define $P_1:=Q_n$, where $P_1=\setDef{(\alpha,x)\in \R\times\R^{n-1}}{(\alpha,x)\in Q_n}$ and $P_2\subseteq \R\times \R^2$ to be the simplex $T$ with the vertex set $\{(0,0,0), (0,1,1), (1,0,1), (1,1,0)\}$, where $P_2=\setDef{(\beta,y)\in \R\times\R^2}{(\beta,y)\in T}$, and let the function $f:\R\times\R\rightarrow \R\cup\{\infty\}$ be as follows
\begin{equation*}
 f(\alpha,\beta):=\begin{cases}
                   \alpha&\text{if}\quad \alpha=\beta\\
		   \infty&\text{otherwise}
                  \end{cases}\,.
\end{equation*}
In this case,  the polytope $P$ defined in \eqref{prob:Margot} equals the convex hull of the following points
\begin{multline*}
    (\gamma,x,y)\in\R\times\R^{n-1}\times \R^2\\ \text{ such that }(\alpha,x)\in \vertex(Q_n)\,, (\beta ,y)\in \vertex(T)\,,\\\gamma =f(\alpha,\beta )\text{  and  } \gamma \neq \infty \text{ for some  } \alpha\in\R, \beta \in\R\,.
\end{multline*}
It is straightforward to see that $p_{x,y}(P)$ equals the parity polytope $Q_{n+1}$.

Let us check the conditions of  Theorem~\ref{thm:margot_theorem} corresponding to the above polytopes $P_1$, $P_2$ and the function $f$. The projections of $Q_n$ and $T$ in the first variable satisfy the PF-property. Indeed, the projections for both $Q_n$ and $T$ form a simplex of dimension one, namely the interval with endpoints $(0)$ and $(1)$. 
The polytope $P_3$ is the convex hull of two points $(0,0,0)$ and $(1,1,1)$, i.e. 
$$P_3=\setDef{(\gamma,\alpha,\beta)\in \R\times\R\times\R}{\gamma=\alpha=\beta,\, 0\le\beta\le 1},$$  and thus every vertex of $P_3$ is projected into a vertex of $p_{\gamma}(P_3)$.

Hence, all conditions of Theorem~\ref{thm:margot_theorem} hold. Thus, $Q:=\{(\gamma,\alpha,x,\beta, y)\in  \R\times\R\times\R^{n-1}\times \R\times\R^{2}:\;A_1(\alpha,x)\le b_1, A_2(\beta,y)\le b_2, A_3 (\gamma,\alpha, \beta)\le b_3\}$ together with the projection $p_{\gamma,x,y}$ forms an extension of $P$; subsequently, $Q$ and $p_{x,y}$ form an extension of $Q_{n+1}$.  Note, that the system $A_1(\alpha,x)\le b_1$ may be replaced by an extended formulation for $P_1$; $A_2(\beta,y)\le b_2$ is the description of the simplex $T$, i.e. $A_2(\beta,y)\le b_2$ may be assumed to consist of four inequalities. Moreover, in the definition of $Q$ the constraints $A_3 (\gamma,\alpha, \beta)\le b_3$ may be replaced by the constraint $\gamma=\alpha=\beta$, since the inequalities $0\le \beta\le 1$ are implied by the system $A_2(\beta,y)\le b_2$ describing the simplex $T$.

Hence, the polytope $Q$ has an extended formulation with $s+4$ inequalities whenever $Q_n$ has an extended formulation with $s$ inequalities. This fact leads to an extended formulation for $Q_{n+1}$ of size~$4n$, since the polytope $Q_1=\{(0)\}$ can be described by equations, i.e. does not need inequalities in its description.

\subsection{Stable Set Polytope}\label{example:stable_set}
 Let us consider the stable set polytope~$\stab(G)$ for a graph~$G(V,E)$ with a cutset~$U\subseteq V$. Let the  graph~$G$ be decomposed  via the cutset~$U$ into graphs~$G_1(V_1\cup U,E_1)$, $G_2(V_2\cup U,E_2)$ and let~$f:\R^{|U|}\times \R^{|U|}\rightarrow\R^{|U|}\cup\{\infty\}$ be defined as follows
\begin{equation*}
 f(\alpha,\beta):=\begin{cases}
                   \alpha&\text{if}\quad \alpha=\beta\\
		   \infty&\text{otherwise}
                  \end{cases}\,.
\end{equation*}
Then the stable set polytope~$\stab(G)$ of the graph~$G$ is the convex hull of the points
\begin{multline*}
    (\gamma ,x,y)\in\R^{|U|}\times \R^{|V_1|}\times\R^{|V_2|}\\ \text{ such that }(\alpha,x)\in \vertex(\stab(G_1))\,, (\beta ,y)\in \vertex(\stab(G_2))\,,\\\gamma =f(\alpha,\beta )\text{  and  } \gamma \neq \infty \text{ for some  } \alpha\in\R^{|U|}, \beta \in\R^{|U|}\,.
\end{multline*}

The polytope $\stab(G[U])$, where $G[U]$ is the subgraph induced by $U$, is the projection of $\stab(G)$ into the space indexed by the  node set~$U$.  Since~$\stab(G)$  is a~$0,1$-polytope, the vertices of~$\stab(G)$ project onto vertices of~$\stab(G[U])$.

If~$G[U]$ is a clique,   $\stab(G[U])$ is the convex hull of~$\bf 0$ and the $|U|$~unit vectors in $\R^{|U|}$. Hence~$\stab(G[U])$ is a simplex, i.e.  every subset of vertices of~$\stab(G[U])$ is a face of~$\stab(G[U])$. This shows that every face of~$\stab(G)$ projects to a face of~$\stab(G[U])$. Hence the polytope~$\stab(G)$ and the projection into the space indexed by the  node set~$U$ have the PF-property.

For the stable set polytope, Theorem~\ref{thm:margot_theorem} implies the following  result of Chv\'atal~\cite{CHV}: Given a graph~$G(V,E)$ and a clique cutset~$U$ decomposing $G$ into~$G_1$ and~$G_2$ a linear description for~$\stab(G)$ can be obtained from linear descriptions for the polytopes~$\stab(G_1)$ and~$\stab(G_2)$ by associating the variables indexed by~$U$.
\smallskip
\begin{remark}
It is not hard to prove that for an induced subgraph~$G'(V',E')$ of the graph~$G(V,E)$: the polytope~$\stab(G)$ and the projection into the space indexed by the  node set~$V'$ has the PF-property if and only if two nodes in~$V'$ are adjacent whenever there is a path in~$G$ between these nodes with all inner nodes in~$V\backslash V'$. To show the "only if" direction consider a face of the polytope~$\stab(G)$ induced by the shortest path between arbitrary two non adjacent nodes in~$V'$ with all inner nodes in~$V\backslash V'$: the sum of two variables for every edge in the path equals~$1$, the variables indexed by~$V'$ except the endpoints of the path equal~$0$. In this case the defined face projects on a diagonal of a 2-dimensional face of the polytope~$\stab(G')$ which is a square. The other direction is implied by Theorem~\ref{thm:margot_theorem} and Proposition~\ref{prop:pfp_property_intersection_preserves_pfp}.
\end{remark}

\subsection{Travelling Salesman Polytope}

Let $G$ be a connected graph and $H=\delta(S)$, $S \subseteq V$ be a cutset of the graph. Let $G_1(V_1, E_1\cup H)$, $G_2(V_2, E_2\cup H)$ be obtained from $G$ by contracting $S$, $V\setminus S$ respectively into a single vertex.

If the cutest $H$ consists of at most three edges the travelling salesman polytope~$\ts(G)$ of the graph~$G$ is the convex hull of the points
\begin{multline*}
    (\gamma ,x,y)\in\R^{|H|}\times \R^{|E_1|}\times\R^{|E_2|}\\ \text{ such that }(\alpha,x)\in \vertex(\ts(G_1))\,, (\beta ,y)\in \vertex(\ts(G_2))\,,\\\gamma =f(\alpha,\beta )\text{  and  } \gamma \neq \infty \text{ for some  } \alpha\in\R^{|H|}, \beta \in\R^{|H|}\,,
\end{multline*}
where~$f:\R^{|H|}\times \R^{|H|}\rightarrow\R^{|H|}\cup\{\infty\}$ is defined as
\begin{equation*}
 f(\alpha,\beta):=\begin{cases}
                   \alpha&\text{if}\quad \alpha=\beta\\
		   \infty&\text{otherwise}
                  \end{cases}\,.
\end{equation*}

If $|H|\le 3$, the projection of $\ts(G_1)$, $\ts(G_2)$ into the space indexed by the edge set~$H$ is a simplex. Indeed, in this case every hamiltonian tour in $G_1$, $G_2$ uses exactly two edges of $H$, and thus  the projection of $\ts(G_1)$, $\ts(G_2)$ into the space indexed by the edge set~$H$ is the convex hull of $0,1$-points with at most three coordinates exactly two of which should be equal to $1$.

For cutsets of at most three edges, Theorem~\ref{thm:margot_theorem} immediately leads to the formulations of the travelling salesman polytope for Halin and prismatic graphs provided in \cite{CNP}

\begin{remark}

In all examples above, the function $f:(\alpha,\beta)\mapsto \gamma$ was such that $\gamma=\infty$ whenever $\alpha$ and $\beta$ were distinct, and $\gamma=\alpha$ otherwise. However, in the dissertation~\cite{FM}, Margot studied extended formulations for strongly connected subgraphs in a series-parallel graph, where the series composition rule gives a more complicated composition function $f$.
\end{remark}

\section{A Characterization of Projected Faces Property}\label{sec:characterization}
  Theorem \ref{thm:pfp_property_main} below provides a characterization of the PF-property. The "only if" direction was proved by Margot in his dissertation \cite{FM} and is fundamental for the proof of Theorem \ref{thm:margot_theorem}. We prove the "if" direction and provide a compact proof of the other direction, the one proved by Margot.

\begin{lem}\label{lem:auxiliary_thm_pfp_property_main}
Given a polytope~$P\subseteq\R^n\times\R^d$ and its facet $F$ one of the following holds:
\begin{enumerate}[(i)]
\item $\dim (p_x(F))=\dim(p_x(P))$
\item $\dim (p_x(F))=\dim(p_x(P))-1$ and $p_x(F)$ is a facet of $p_x(P)$.
\end{enumerate}
\end{lem}
\begin{proof} Let $v$ be a vertex of $P$ not lying on the facet $F$. The vertex $v$ together with $F$ affinely generate $\aff(P)$, in particular $p_x(v)$ together with $p_x(F)$ affinely generate $\aff(p_x(P))$. Thus, $\dim (p_x(F))=\dim(p_x(P))$ or $\dim (p_x(F))=\dim(p_x(P))-1$. It remains to show that $p_x(F)$ is a face of $p_x(P)$ whenever $\dim (p_x(F))=\dim(p_x(P))-1$.

Let $\dim (p_x(F))=\dim(p_x(P))-1$. Hence, there is a hyperplane $H$ in the affine space $\aff(p_x(P))$ containing $p_x(F)$. Thus, the hyperplane $p_x^{-1}(H)\cap \aff(P)$ is the unique hyperplane in the affine space $\aff(P)$ containing the facet $F$ of $P$. Thus, $F=p_x^{-1}(H)\cap P$ implying $p_x(F)=H\cap p_x(P)$. On the other hand since $p_x^{-1}(H)\cap \aff(P)$ is a facet defining hyperplane for $P$, the hyperplane $H$ is a face defining for the polytope $p_x(P)$, i.e. $p_x(F)=H\cap p_x(P)$ is a face of the polytope $p_x(P)$.
\end{proof}

\begin{thm}\label{thm:pfp_property_main}
Given a polytope~$P\subseteq\R^n\times\R^d$ and the projection in the first~$n$ coordinates~$p_x:\R^n\times\R^d\rightarrow\R^n$,  the pair~$(P,p_x)$ has the PF-property if and only if for every~$S\subseteq\vertex(p_x(P))$
%\footnote{The condition~$S\subseteq\vertex(p_x(P))$ can be changed to the condition~$S\subseteq p_x(P)$. Moreover this modification would even lead to some simplifications in the proof. Nevertheless we decided to keep the above condition to be closer to the original formulation of Margot~\cite{FM}.}
 the following holds (see Figure~\ref{fig:characterization})
%\footnote{ It may be helpful to provide the following geometric interpretation of the equation~\eqref{eq:pfp_main_property}: The polytope~$\setDef{v\in P}{p_x(v)\in\conv(S)}$ is the convex hull of the fibers~$p_x^{-1}(x)\cap P$ for~$x\in S$.}
  \begin{equation}\label{eq:pfp_main_property}
      \conv\setDef{v\in\vertex(P)}{p_x(v)\in S}=\setDef{v\in P}{p_x(v)\in\conv(S)}\,,
  \end{equation}
  i.e. the pair~$(P,p_x)$ has the PF-property if and only if for every~$S\subseteq\vertex(p_x(P))$
    \begin{equation}
     p_x^{-1}(\conv(S))\cap P=\conv(p_x^{-1}(S)\cap\vertex(P))\,.
  \end{equation}
\end{thm}
\begin{figure}[H]
  \begin{center}
    \begin{tikzpicture}[scale=1.4]
      \foreach \x/\y in {0/0,70/1,100/2,180/3,220/4,300/5} {\coordinate (R-\y) at ({cos(\x)}, {sin(\x)*.3+.5*cos(\x)+.7});};
      \foreach \y  in {0,...,5} { \pgfmathparse{mod(\y+1,6)} \draw (R-\y)--(R-\pgfmathresult);}
      \foreach \x/\y in {0/0,70/1,100/2,180/3,220/4,300/5} {\coordinate (L-\y) at ({cos(\x)}, {sin(\x)*.3-2});};
      \foreach \y  in {0,...,5} { \pgfmathparse{mod(\y+1,6)} \draw (L-\y)--(L-\pgfmathresult); \draw (L-\y)--(R-\y);}
      \draw[dashed, white] (L-1)--(R-1);
      \draw[dashed, white] (L-2)--(R-2);
      \draw[dashed, white] (L-0)--(L-1)--(L-2)--(L-3);
      \draw[fill=blue!40, fill opacity=0.5] (L-2)--(R-2)--(R-5)--(L-5)--(L-2);
      \foreach \y  in {2,5}{
     		\node[circle,shading=ball, scale=0.5] at (L-\y) {};
		\node[circle,shading=ball, scale=0.5] at (R-\y) {};
      }
    \end{tikzpicture}
    \hspace{1cm}
    \begin{tikzpicture}[scale=1.4]
   	\draw [white,name path=line1] ({cos(100)}, {sin(100)*.3+1})--({cos(300)}, {sin(300)*.3+1});
	\draw [white,name path=line2] ({cos(220)}, {sin(220)*.3+1})--({cos(0)}, {sin(0)*.3+1});
	\path [name intersections={of=line1 and line2,by=K-2}];
      \foreach \x/\y in {0/0,70/1,100/2,180/3,220/4,300/5} {\coordinate (R-\y) at ({cos(\x)}, {sin(\x)*.3});};
      \foreach \x/\y in {0/0,100/2,220/4} {\coordinate (R-\y) at ({cos(\x)}, {sin(\x)*.3+1});};
      \foreach \y  in {0,...,5} { \pgfmathparse{mod(\y+1,6)} \draw (R-\y)--(R-\pgfmathresult);}
      \draw[very thick] (R-0)--(R-2)--(R-4)--(R-0);
      \foreach \x/\y in {0/0,70/1,100/2,180/3,220/4,300/5} {\coordinate (L-\y) at ({cos(\x)}, {sin(\x)*.3-2});};
      \foreach \y  in {0,...,5} { \pgfmathparse{mod(\y+1,6)} \draw (L-\y)--(L-\pgfmathresult); \draw (L-\y)--(R-\y);
      \draw[dashed, white] (L-1)--(R-1);
      \draw[dashed, white] (L-2)--(R-2);
      \draw[dashed, white] (L-0)--(L-1)--(L-2)--(L-3);
      \draw[dashed, white] (R-0)--(R-1)--(R-2)--(R-3);
      \draw[fill=blue!40, fill opacity=0.1] (R-2)--(L-2)--(L-5)--(R-5)--(K-2)--(R-2);
       \foreach \y  in {2,5}{
     		\node[circle,shading=ball, scale=0.5] at (L-\y) {};
		\node[circle,shading=ball, scale=0.5] at (R-\y) {};
      }
%     \node[circle,shading=ball, scale=0.5] at (K-2) {};
%     \node[circle,shading=ball, scale=0.5] at (K-4) {};
}
    \end{tikzpicture}
  \end{center}
   \caption{Here, the vertices $S$ are the marked vertices in the lower facets. On the left, the convex hull of all vertices of the polytope $P$ with projection in $S$ is the set of points in $P$ which project into $\conv(S)$. On the right, the convex hull of all vertices of the polytope $P$ with projection in $S$ is strictly contained in the set of all points in $P$ which project into $\conv(S)$. }
  \label{fig:characterization}
\end{figure}
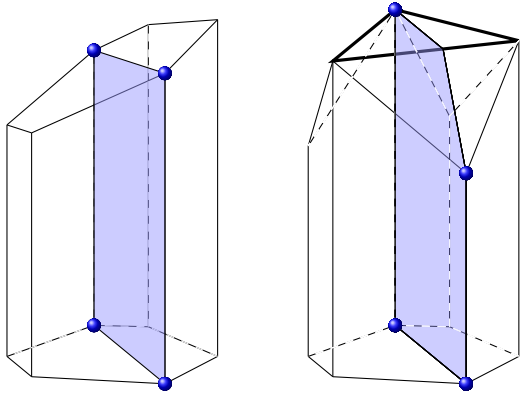

%In his dissertation Fran\c{c}ois Margot showed that for every pair~$(P,p_x)$ with %the PFP-property the statement~\eqref{eq:pfp_main_property} holds.
\begin{proof}
Part "only if"~\cite{FM}: Let~$Q\subseteq\R^n\times\R^d$ be the polytope
\begin{equation*}
    \setDef{v\in P}{p_x(v)\in\conv(S)}\,.
\end{equation*}
Then  the inclusion~$\conv\setDef{v\in\vertex(P)}{p_x(v)\in S}\subseteq Q$ is immediate. Thus, it is enough to show that every vertex of~$Q$ is a vertex of the polytope~$P$ whose projection is in~$S$.

The polytope $Q$ is the intersection of the polytope $P$ and the polyhedron $\setDef{v}{p_x(v)\in\conv(S)}$. Thus, every vertex $v\in\vertex(Q)$ is induced by a face~$F\subseteq\R^{n}\times\R^{d}$ of the polytope~$P$ and by a face~$F'\subseteq\R^n$ of the polytope~$\conv(S)$ in the following manner
\begin{equation*}
    \{v\}=F\cap\setDef{(x,y)\in\R^{n}\times\R^{d}}{x\in F'}\,,
\end{equation*}
which implies
\begin{equation*}
    \{p_x(v)\}=p_x(F\cap \setDef{(x,y)\in\R^{n}\times\R^{d}}{x\in F'})=p_x(F)\cap F'\,.
\end{equation*}
 Since the pair~$(P,p_x)$ has the PF-property, $p_x(F)$ is a face of the polytope~$p_x(P)$, and thus~$\{p_x(v)\}=p_x(F)\cap F'$ is a face of the polytope $\conv(S)$, i.e.~$p_x(v)$ is a point in~$S$ and thus a vertex of $p_x(P)$. Finally, the set~$\{v\}=F\cap p_x^{-1}(p_x(v))$ is the intersection of the face $F$ of $P$ and the face $P\cap p_x^{-1}(p_x(v))$ of $P$, i.e. $\{v\}$ is a face of the polytope~$P$ and hence~$v$ is a vertex of $P$.

\bigskip

Part "if":  Assume for the sake of contradiction that $(P,p_x)$ does not have the PF-property and let $F$ be an inclusion wise maximal face of~$P$ such that~$p_x(F)$ is not a face of the polytope~$p_x(P)$.

Then all inclusion wise minimal faces $F'$ of $P$, which properly contain $F$, satisfy the equation $\dim(p_x(F))=\dim(p_x(F'))$. Indeed, due to the inclusion wise maximality of $F$ the projection $p_x(F')$ is a face of $p_x(P)$. The projection $p_x(F)$ is not a face of $p_x(F')$ since otherwise $p_x(F)$ is also a face of $p_x(P)$.  By Lemma~\ref{lem:auxiliary_thm_pfp_property_main}, applied to the polytope $F'$ and its facet $F$, we obtain that $\dim(p_x(F))=\dim(p_x(F'))$. Moreover, since the faces $F'$ all together affinely span $\aff(P)$  and for all $F'$ holds the equation $\dim(p_x(F))=\dim(p_x(F'))$, we obtain $$p_x(\aff(P))= p_x(\aff(\cup_{F'}F'))=\aff(\cup_{F'}p_x(F'))=\aff(p_x(F))\,.$$
%Let us prove that~$\dim(p_x(F))$ equals~$\dim(p_x(P))$. Indeed, if~$\dim(p_x(F))$ is smaller than~$\dim(p_x(P))$, consider the inclusion wise minimal faces of~$P$ which strictly contain~$F$. These faces taken together affinely span the affine hull of the polytope~$P$, thus their projections taken together affinely span the affine hull of the polytope~$p_x(P)$. Hence there is a face~$F'$ of~$P$ such that $F\subseteq F'$, $\dim(p_x(F'))=\dim(p_x(F))+1$ and~$\dim(F')=\dim(F)+1$. Thus~$p_x(F)$ is a facet of~$p_x(F')$. Moreover~$p_x(F')$ is not a face of the polytope~$p_x(P)$, because its face~$p_x(F)$ is not a face of the polytope~$p_x(P)$, but this contradicts the choice of~$F$ as an inclusion wise maximal face of~$P$ such that~$p_x(F)$ is not a face of the polytope~$p_x(P)$.

  Hence, $\dim(p_x(F))=\dim(p_x(P))$. Therefore $\vertex(p_x(F))$ contains a set $W$, such that every point in~$p_x(P)$ has a unique representation as an affine combination of points in~$W$.
Let $w$ be a vertex of $p_x(P)$, which is not in $p_x(F)$. By Radon's Theorem~\cite{R}, the set of points~$W\cup\{w\}$ can be partitioned in two sets~$W_1\subseteq W$ and $W_2$ such that there is a point~$u$ which lies in~$\conv(W_1) \cap\conv(W_2)$. But this contradicts the statement~\eqref{eq:pfp_main_property} for the set~$S=W_2$ since
  \begin{multline*}
      u \in p_x(\setDef{v\in P}{p_x(v)\in\conv(S)}\cap F)\qquad\text{and}\\
      u \not\in p_x(\conv\setDef{v\in\vertex(P)}{p_x(v)\in S}\cap F)=\\p_x(\conv\setDef{v\in\vertex(F)}{p_x(v)\in S})\,.
  \end{multline*}
 To finish the proof it is enough to notice that $W_2$ is a subset of~$\vertex(p_x(P))$. Indeed every vertex of the polytope~$P$ is projected to a vertex of the polytope~$p_x(P)$, otherwise we would be able to choose~$S$ equal to~$\vertex(p_x(P))$, which leads to
\begin{equation*}
     \conv\setDef{v\in\vertex(P)}{p_x(v)\in S}\subsetneq P=\setDef{v\in P}{p_x(v)\in\conv(S)}\,.
\end{equation*}
\end{proof}

\section{Linear Descriptions based on the Projected Faces Property}\label{sec:framework}

 Here we prove Theorem~\ref{thm:margot_theorem}, the main result in Margot's dissertation \cite{FM}.

\begin{lem}~\cite{FM}\label{lem:pfp_property_prod}
   The PF-property is compatible with the cartesian product operation, i.e. if a pair $(P_1, p_{\alpha})$ has the PF-property and a pair $(P_2,p_{\beta})$ has the PF-property then the pair $(P_1\times P_2, p_{\alpha,\beta})$ has the PF-property.
\end{lem}
\begin{proof}
Indeed, every face $F$ of $P_1\times P_2$ may be written as $F_1\times F_2$, where $F_1$ is a face of $P_1$ and $F_2$ is a face of $P_2$. Due to the PF-property of $(P_1, p_{\alpha})$ and $(P_2, p_{\beta})$, the projection  $p_{\alpha,\beta}(F_1\times F_2)=p_\alpha(F_1)\times p_\beta(F_2)$ is the cartesion product of a face of $p_\alpha(P_1)$ and a face of $p_\beta(P_2)$, i.e. $p_{\alpha,\beta}(F)=p_{\alpha,\beta}(F_1\times F_2)$ is a face of $p_\alpha(P_1)\times p_\beta(P_2)=p_{\alpha,\beta}(P_1\times P_2)$.
\end{proof}

\begin{prop}~\cite{FM}\label{prop:pfp_property_intersection_preserves_pfp}
  If a pair~$(P,p_x)$ has the PF-property then for every $S\subseteq\vertex(p_x(P))$ the pair
  \begin{equation*}
    (\,P\cap p_x^{-1}(\conv(S)),p_x)
   \end{equation*}
  has the PF-property.
\end{prop}
\begin{proof}
  Indeed, every face of the polytope~$P\cap p_x^{-1}(\conv(S))$ is induced by a face~$F\subseteq\R^{n}\times\R^{d}$ of the polytope~$P$ and by a face~$F'\subseteq\R^n$ of the polytope~$\conv(S)$ in the following manner
  \begin{equation*}
      F\cap\setDef{(x,y)\in\R^{n}\times\R^{d}}{x\in F'}\,,
  \end{equation*}
  and thus its projection equals~$p_x(F)\cap F'$. Due to the PF-property $p_x(F)$ is a face of the polytope~$p_x(P)$, and thus~$p_x(F)\cap F'$ is a face of~$\conv(S)$, i.e. is a face of the polytope~$p_x(P\cap p_x^{-1}(\conv(S)))$.
\end{proof}

\bigskip

Now let us prove Theorem~\ref{thm:margot_theorem}.
\begin{proof}
 From Lemma~\ref{lem:pfp_property_prod} applied to~$(P_1, p_{\alpha})$, $(P_2, p_{\beta})$ and $(p_{\gamma}(P_3),p_{\gamma})$ it follows that the polytope
\begin{multline*}
 p_{\gamma}(P_3)\times P_1\times P_2=\{(\gamma,\alpha,x,\beta, y)\in \R^n\times\R^{n_1}\times\R^{d_1}\times \R^{n_2}\times\R^{d_2}:\\
  \;A_1(\alpha,x)\le b_1, A_2(\beta,y)\le b_2, \gamma \in p_{\gamma}(P_3)\}
\end{multline*}
together with the projection $p_{\gamma,\alpha,\beta}$ has the PF-property.

Let $S$ be defined as the following set of points
\begin{multline*}
    (\gamma,\alpha,\beta)\in \R^n\times\R^{n_1}\times\R^{n_2} \text{ such that } \\\alpha \in\vertex(p_{\alpha}(P_1))\, , \beta \in\vertex(p_{\beta}(P_2))\,,\gamma=f(\alpha,\beta)\text{  and  } \gamma \neq \infty\,.
\end{multline*}
Since $S=\vertex(P_3)$ and every vertex of the polytope~$P_3$ is projected to a vertex of the polytope~$p_{\gamma}(P_3)$ via the map~$p_{\gamma}$, we obtain  $S\subseteq \vertex(p_{\gamma}(P_3)) \times \vertex(p_{\alpha}(P_1))\times\vertex(p_{\beta}(P_2))$. Hence $S\subseteq\vertex(p_{\gamma,\alpha,\beta}(p_{\gamma}(P_3)\times P_1\times P_2))$.

From Theorem~\ref{thm:pfp_property_main} applied to the pair $(p_{\gamma}(P_3)\times P_1\times P_2, p_{\gamma,\alpha,\beta})$ and the vertex set~$S$ it follows that the polytope
\begin{multline*}
 \{(\gamma,\alpha,x,\beta, y)\in \R^n\times \R^{n_1}\times\R^{d_1}\times \R^{n_2}\times\R^{d_2}:\\
  \;A_1(\alpha,x)\le b_1, A_2(\beta,y)\le b_2, \gamma \in p_{\gamma}(P_3), A_3(\gamma,\alpha,\beta)\le b_3\}
\end{multline*}
together with the projection map~$p_{\gamma,x,y}$ provides an extended formulation of the polytope~$P$. To finish the proof note that the condition~$\gamma\in p_{\gamma}(P_3)$ is implied by the system~$A_3(\gamma,\alpha,\beta)\le b_3$, and thus the obtained extension of the polytope~$P$ is formed by the polytope~$Q$ together with the projection map~$p_{\gamma,x,y}$, which shows~\eqref{enum:projection}. On the other side from Proposition~\ref{prop:pfp_property_intersection_preserves_pfp} applied to $(p_{\gamma}(P_3)\times P_1\times P_2,p_{\gamma,\alpha,\beta})$ and the above set~$S$ we may conclude that the pair~$(Q,p_{\gamma,\alpha,\beta})$ has the PF-property showing~\eqref{enum:pfp}.
\end{proof}

In order to use this idea iteratively, we would need to show that every time Theorem~\ref{thm:margot_theorem} is applied the pairs~$(P_1,p_{\alpha})$, $(P_2,p_{\beta})$ have the PF-property. To guarantee that these pairs have the PF-property, we may consider the simplest case, i.e. the case when the polytopes~$p_{\alpha}(P_1)$, $p_{\beta}(P_2)$ are simplices and every vertex of~$P_1$ and $P_2$ is projected to a vertex of~$p_{\alpha}(P_1)$ and $p_{\beta}(P_2)$, respectively. Actually, in the dissertation Margot restricted his attention to these cases.

\begin{remark}
One may conjecture that a pair~$(P,p_x)$ has the PF-property if and only if every vertex of the polytope~$P$  is projected to a vertex of the polytope~$p_x(P)$ and every facet of the polytope~$P$ is projected to a face of the polytope~$p_x(P)$, i.e. it is enough that every facet and every vertex of the polytope~$P$ projects to a face of~$p_x(P)$.

    It is easy to see that this conjecture is true for all polytopes~$P$ of dimension at most three. However, this conjecture is not true in general: given a number~$k$ and a $(k+1)$-neighborly polytope~$P'\subseteq\R^n$ (which is not a simplex) let the polytope~$P$ be a simplex with $(k+1) |\vertex(P')|$ vertices such that the polytope~$P'$ is the projection of~$P$ in the last $n$~coordinates and for each vertex of~$P'$ there are exactly $k+1$~vertices of~$P$ projected to it. Then the pair~$(P,p_x)$ does not possess the PF-property, but every face of dimension or of co-dimension at most~$k$ is projected to a face of the polytope~$p_x(P)$ (the projection of the faces of co-dimension at most~$k$ is equal $p_x(P)$ and the faces of dimension at most~$k$ contain at most~$k+1$ vertices and thus are projected to faces of the polytope $p_x(P)$).

\end{remark}

\section{Affinely Generated Polyhedral Relations}\label{sec:affine_generated_relations}

Let~$P\subseteq\R^{n}\times\R^{d}$ be a polytope and $p_x:\R^{n}\times\R^{d}\rightarrow \R^n$ be the projection map in the first $n$~coordinates. The polytope~$P$ is called a {\em polyhedral relation of type~$(n,d)$ generated by affine maps~$\rho_1$, \ldots, $\rho_t: \R^n \rightarrow \R^d$} if for every point~$x\in p_x(P)$
\begin{equation}\label{eq:affinely_generated_def}
    P\cap p_x^{-1}(x)=\conv \setDef{(x,\rho_j(x))}{1 \le j \le t}\,.
\end{equation}

For example consider
\begin{multline*}
    P^*:=\{(x,y)\in\R\times\R^2\,:\, x-y_1-y_2\le 0,\,\\ y_1+y_2+x\le 2,\,y_1-y_2-x\le 0,\, y_2-y_1-x\le 0\}\,.
\end{multline*}
It is not hard to see that~$p_x(P^*)=[0,1]$ and that the polytope~$P^*$ is a polyhedral relation of type~$(1,2)$, which is affinely generated by~$\rho_1:x\mapsto(x,0)$, $\rho_2:x\mapsto(0,x)$, $\rho_3:x\mapsto(1-x,1)$ and $\rho_4:x\mapsto(1,1-x)$.
\smallskip

Polyhedral relations were introduced in \cite{KP} as a paradigm to construct extended formulations. We illustrate it with the parity polytope from Section~\ref{sec:examples}.
To construct an extended formulation of the parity polytope~$Q_n$ let us define
\begin{equation*}
    Q^*_n:=\setDef{(x,y,z)\in \R\times\R^2\times\R^{n-2}}{(x,z)\in Q_{n-1},\, (x,y)\in P^*}\,,
\end{equation*}
then due to the fact that~$P^*$ is a polyhedral relation, which is affinely generated by the maps~$\rho_1$, $\rho_2$, $\rho_3$ and $\rho_4$, we obtain
\begin{multline*}
    Q^*_n:=\conv\{(x,y,z)\in \R\times\R^2\times\R^{n-2}\,:\,(x,z)\in \vertex(Q_{n-1}),\\ y\in\conv(\rho_1(x),\rho_2(x),\rho_3(x),\rho_4(x))\}\,.
\end{multline*}
The last equation leads to the fact that~$p_{y,z}(Q^*_n)$ equals~$Q_n$. Applying this iteratively, we obtain the extended formulation from  Section~\ref{sec:examples} for the polytope~$Q_n$ of size~$4(n-1)$, since the condition~$(x,z)\in Q_{n-1}$ for every~$n\ge 2$ in the definition of the polytope~$Q^*_n$ can be reformulated through the polytope~$Q^*_{n-1}$ and for $n=1$ the polytope~$Q_n\subseteq \R$ consists of a single point~$(0)$.

For a detailed discussion on polyhedral relations we refer the reader to the paper~\cite{KP}.

\begin{thm}\label{thm:pfp_isomorph_faces}
   For a polytope~$P\subseteq\R^n\times\R^d$ and the projection~$p_x:\R^n\times\R^d\rightarrow\R^n$ the pair~$(P,p_x)$ has the PF-property if and only if the polytope~$P$ is an affinely generated polyhedral relation of type~$(n,d)$.
\end{thm}
\begin{proof}
Part "only if":  Consider a point~$x$ in the relative interior of~$p_x(P)$. For every vertex~$v$ of the fiber~$P\cap p_x^{-1}(x)$ there is a face~$F$ of the polytope~$P$ such that~$v=F\cap p_x^{-1}(x)$. Let us prove $\dim(F)=\dim(p_x(P))$. From the equations~$\dim(F)=\dim(F\cap p_x^{-1}(x))+\dim(p_x(F))$ and~$\dim(F\cap p_x^{-1}(x))=0$ follows that~$\dim(F)$ is at most~$\dim(p_x(P))$. Moreover, the projection $p_x(F)$ of the face~$F$ equals the polytope~$p_x(P)$, since the pair~$(P,p_x)$ has the PF-property and the point~$x$ lies in the relative interior of the polytope~$p_x(P)$. Let us consider $\dim(p_x(P))$-dimensional faces~$F_i\subseteq\R^n\times\R^d$, $1\le i \le t$ of the polytope~$P$, which are projected to~$p_x(P)$. The projection~$p_x$ defines an affine isomorphism between a face~$F_i$, $1\le i \le t$   and the polytope~$p_x(P)$. Let us represent the inverse map to this isomorphism in the following form~$(\id,\rho_i): \R^n \rightarrow \R^n \times \R^d$, where~$\id:\R^n\rightarrow\R^n$ denotes the identity map and $\rho_i$ is an affine map. Hence every vertex~$v$ of the fiber~$P\cap p_x^{-1}(x)$ equals~$(x,\rho_i(x))$ for some~$i$, $1\le i\le t$. Now it is straightforward to see that for the point~$x$ the equation~\eqref{eq:affinely_generated_def} holds.

 Now let us consider a point~$x\in p_x(P)$, which is not in the relative interior of the polytope~$p_x(P)$. If there is a point~$v$ in the fiber~$P\cap p_x^{-1}(x)$, which does not lie in~$\conv \setDef{(x,\rho_i(x))}{1 \le i \le t}$, then from the continuity reasons there is a point~$x'$ in the relative interior of~$p_x(P)$ and a point~$v'\in P\cap p_x^{-1}(x')$, such that~$v'$ does not lie in~$\conv \setDef{(x',\rho_i(x'))}{1 \le i \le t}$, which contradicts the result of the discussion above. Finally, since the point~$(x,\rho_i(x))$ lies in~$P\cap p_x^{-1}(x)$ for every~$1\le i\le t$, the equation~\eqref{eq:affinely_generated_def} holds also for all points~$x$ outside the relative interior of the polytope~$p_x(P)$. Thus the polytope~$P$ is a polyhedral relation of type~$(n,d)$, which is affinely generated by maps~$\rho_1$, \ldots, $\rho_t: \R^n \rightarrow \R^d$.

   Part "if":   Since~$P$ is a polyhedral relation of type~$(n,d)$ generated by affine maps~$\rho_1$, \ldots, $\rho_t: \R^n \rightarrow \R^d$, then for every~$S\subseteq \vertex(p_x(P))$ the following equation holds
    \begin{multline*}
    \conv\setDef{v\in\vertex(P)}{p_x(v)\in S}=\conv\setDef{v\in P}{p_x(v)\in S}=\\
      \conv\setDef{(x,\rho_j(x))}{x\in S,\, 1 \le j \le t }=\\
      \conv\setDef{(x,\rho_j(x))}{x\in \conv(S),\, 1 \le j \le t }=\\
      \conv\setDef{v\in P}{p_x(v)\in \conv(S)}\,.
    \end{multline*}
   Thus the statement follows from Theorem~\ref{thm:pfp_property_main}.

\end{proof}

The "only if" part of the Theorem~\ref{thm:pfp_isomorph_faces} follows from the characterization of the affinely generated polyhedral relations in~\cite{KP}. The above "only if" part of the proof is a variant of the proof for this characterization in~\cite{KP}. Additionally, from the proof of the "only if" part in Theorem~\ref{thm:pfp_isomorph_faces} we obtain also the next observation.

\begin{obs}
  For every pair $(P,p_x)$ with the PF-property the number of facets of the polytope~$p_x(P)$ is bounded from above by the number of facets of the polytope~$P$.
\end{obs}
\begin{proof}
  It is enough to notice that there exists a face of the polytope~$P$, which is affinely isomorphic to the polytope~$p_x(P)$.
\end{proof}

{\bf Acknowledgment.} We thank the referees for their comments and suggestions, which  increased the readability of the presented paper.


\begin {thebibliography} {1}

\bibitem{CK} R. D. Carr and G. Konjevod, Polyhedral Combinatorics, {\it Tutorials on emerging methodologies and applications in Operations Research}, (2004).


\bibitem{CHV} V. Chv\'atal, On certain polytopes associated with graphs, {\it Journal of Combinatorial Theory B, 18\/} (1975) 138-154.


\bibitem{CCZ} M. Conforti, G. Cornu\'{e}jols and G. Zambelli,  Extended formulations in combinatorial optimization, {\it 4OR} (2010) 1-48.

\bibitem{CNP} G. Cornu{\'e}jols, D. Naddef and W. Pulleyblank, The Traveling Salesman Problem in Graphs with 3-edge Cutsets. {\it J. ACM} (1985) 383--410.


\bibitem{Je} R. Jeroslow,  On defining sets of vertices of the hypercube by linear inequalities, {\it Discrete Mathematics}, (1975) 119--124.


\bibitem{K} V. Kaibel Extended Formulations in Combinatorial Optimization,  {\em Optima 85\/} (2011) 2-7.


\bibitem {KP} V. Kaibel and K. Pashkovich,  Constructing extended formulations from reflection relations, in: {\it Proceedings of the 15th international conference on Integer programming and combinatoral optimization (New York, NY, 2011)} 287--300.


\bibitem{FM} F. Margot,  Composition de polytopes combinatoires: une approche par projection, {\it Thesis (\'{E}cole polytechnique f\'{e}d\'{e}rale de Lausanne, 1994)}.


\bibitem{R} J. Radon, Mengen konvexer K\"{o}rper, die einen gemeinsamen Punkt enthalten, {\it Mathematische Annalen}, (1921) 113-115.



\bibitem{SC} A. Schrijver, {\em Theory of Linear and Integer Programming},
Wiley, New York (1986).



\end{thebibliography}
\end{document}